\documentclass[preprint,12pt]{elsarticle}
\usepackage[english]{babel}

\usepackage{amssymb, mathrsfs, amsmath, amsthm, color}

\theoremstyle{definition}
\newtheorem{definition}{Definition}[section]

\newtheorem{remark}{Remark}[section]
\theoremstyle{plain}
\newtheorem{theorem}{Theorem}[section]
\newtheorem{lemma}{Lemma}[section]
\newtheorem{proposition}{Proposition}[section]
\newtheorem{example}{Example}[section]

\journal{arXiv}

\begin{document}

\begin{frontmatter}

\title{Upper and lower estimates for rate of convergence\\in the Chernoff product formula\\for semigroups of operators}

\author{Oleg E. Galkin and Ivan D. Remizov}

\address{National Research University Higher School of Economics, Russian Federation
	
Laboratory of Dynamical Systems and Applications NRU HSE
	
25/12 Bol. Pecherskaya Ulitsa, Room 412, Nizhny Novgorod, 603155, Russia}

\ead{olegegalkin@ya.ru, ivremizov@yandex.ru}

\begin{abstract}
Chernoff approximations to strongly continuous one-parameter semigroups give solutions to a wide class of differential equations.
This paper studies the rate of convergence of the Chernoff approximations. We provide simple natural examples for which the convergence is arbitrary fast, is arbitrary slow, and holds in the strong operator topology but does not hold in the norm operator topology. We also prove a general theorem that gives an upper estimate for the speed of decay of the norm of the residual term of the Chernoff approximations. The result is applied to one-dimensional parabolic differential equations with variable coefficients. The obtained estimates can be used for the numerical solution of PDEs.\end{abstract}

\begin{keyword}Chernoff product formula \sep approximation of $C_0$-semigroup \sep speed of convergence \sep estimates \sep examples

\textit{MSC 2010:} 47D03; 47D06; 35A35; 41A25

\end{keyword}

\date{2 April 2021, updated 1 November 2021}

\end{frontmatter}

\newpage
\tableofcontents

\section{Introduction}

This paper is devoted to one of the most applicable branch of modern functional analysis, namely to $C_0$-semigroups and their approximations. Three standard textbooks on the topic are \cite{HF, Pazy, EN}; of course, this list is incomplete, but each of these books contains more than enough information on necessary background for the paper. In the paper we do not use any deep results from the $C_0$-semigroup theory and keep all our reasoning very simple and accessible to a broad mathematical audience. We believe that all the text can be understood by everyone who had even an introductive course on functional analysis. 

{\bf Why is this paper interesting and important.} It appeared that our elementary approach allows us to prove the main theorem~\ref{mainth} that develops the result of  the famous Chernoff theorem \cite{Chernoff} on approximations of $C_0$-semigroups. With this new theorem it is possible to find out what one needs to do in order to obtain the so-called fast convergent Chernoff approximations, and what the speed of the convergence can be. These approximations of $C_0$-semigroups provide approximate solutions to the Cauchy problem for a large class of partial differential equations (PDEs), namely linear evolution equations with variable coefficients, such as parabolic or Schr\"odinger equations. This gives a flexible and powerful tool for construction of new numerical methods for solving the Cauchy problem for PDEs. In overview \cite{Butko-2019} one can find many classes of equations for which solution methods based on the Chernoff approximations have been developed, see \cite{OSS2019, ST2020} for most recent applications, see also \cite{R2018, R-2017, RS2018, R-PotAn2020, R5, R2, RemAMC2018, RemJMP}.

Our theorem~\ref{mainth} allows us to prove estimates on the speed of convergence for all of these methods. Moreover, this gives a clue how one can construct Chernoff approximations with faster speed of convergence than other known examples. This is why the results presented in the paper are interesting and important.

{\bf Preliminaries.} Let us recall some relevant notation, definitions, and facts following \cite{EN}.

\begin{definition}\label{semigrdef} Let $\mathcal{F}$ be a Banach space over the field $\mathbb{R}$ or $\mathbb{C}$. Let $\mathscr{L}(\mathcal{F})$ be the set of all bounded linear operators in $\mathcal{F}$. Suppose we have a mapping $V\colon [0,+\infty)\to \mathscr{L}(\mathcal{F}),$ i.e.\ $V(t)$ is a bounded linear operator $V(t)\colon \mathcal{F}\to \mathcal{F}$ for each $t\geq 0.$ The mapping $V$, or equivalently the family $(V(t))_{t\geq 0}$, is called \textit{a strongly continuous one-parameter semigroup of linear bounded operators} (or just \textit{a $C_0$-semigroup}) iff it satisfies the following three conditions: 
	
1) $V(0)$ is the identity operator $I$, i.e. $V(0)\varphi=\varphi$ for each $\varphi\in \mathcal{F}$; 
	
2) $V$ maps the addition of numbers in $[0,+\infty)$ into the composition of operators in $\mathscr{L}(\mathcal{F})$, i.e. for all $t\geq 0$ and all $s\geq 0$ we have $V(t+s)=V(t)\circ V(s),$ where for each $\varphi\in\mathcal{F}$ the notation $(A\circ B)(\varphi)=A(B(\varphi))=AB\varphi$ is used;

3) $V$ is continuous with respect to the strong operator topology in $\mathscr{L}(\mathcal{F})$, i.e. for all $\varphi\in \mathcal{F}$ the function $t\longmapsto V(t)\varphi$, $[0,+\infty)\to \mathcal{F}$  is continuous.
\end{definition}

\begin{remark}
The definition of a \textit{$C_0$-group} $(V(t))_{t\in\mathbb{R}}$ is obtained by the substituting $[0,+\infty)$ with $\mathbb{R}$ in the definition above.
\end{remark}

\begin{definition}\label{defgen}
Let $(V(t))_{t\geq 0}$ be a $C_0$-semigroup in Banach space $\mathcal{F}$. Its \textit{infinitesimal generator} (or just \textit{generator}) is defined as the operator $L\colon D(L)\to\mathcal{F}$ with the domain 
$$
D(L)=\left\{\varphi\in \mathcal{F}: \textrm{there exists a limit }\lim_{t\to +0}\frac{V(t)\varphi-\varphi}{t}\right\} \subset \mathcal{F},
$$ 
and
$$L\varphi=\lim_{t\to +0}\frac{V(t)\varphi-\varphi}{t}.$$ 
Very often the notation $V(t)=e^{tL}$ is used.

If $(V(t))_{t\in\mathbb{R}}$ is a $C_0$-group, then its generator $L$ is defined in the same way: 
$$
D(L)=\left\{\varphi\in \mathcal{F}: \textrm{there exists a limit }\lim_{t\to0}\frac{V(t)\varphi-\varphi}{t}\right\} \subset \mathcal{F},
$$
$$
L\varphi=\lim_{t\to 0}\frac{V(t)\varphi-\varphi}{t}.
$$ 
\end{definition}

\begin{remark}\label{generdef}
It is known that for each $C_0$-semigroup $(V(t))_{t\geq 0}$ in Banach space $\mathcal{F}$, the set $D(L)$ is a dense linear subspace of $\mathcal{F}$ \cite{EN}. 
Moreover, $(L,D(L))$ is a closed linear operator that uniquely defines the semigroup $(V(t))_{t\geq 0}$.
Under condition $V(0)=I$ it is clear that $f\in D(L)$ iff the derivative $\frac{d}{dt}V(t)f\big|_{t=0}$ exists, which is the right derivative in case of a $C_0$-semigroup and two-sided (traditional) derivative in case of a $C_0$-group.
\end{remark}

\begin{definition}\label{remdom}
For a linear operator $A\colon D(A)\to\mathcal{F}$ with the domain $D(A)\subset\mathcal{F}$ and all $n=1,2,3,\ldots$ we define the domain $D(A^n)$ of operator $A^n$ as follows: 
$$(f\in D(A^n))\iff (f\in D(A), Af\in D(A), A^2f\in D(A),\dots, A^{n-1}f\in D(A)),$$
which implies $D(A)\supset D(A^2)\supset\dots\supset D(A^n)$. 
\end{definition}

\begin{definition} \label{defCore}
Let $(A,D(A))$ be a linear operator in Banach space $\mathcal{F}$. Linear subspace $H\subset D(A)$ is called {\it a core} of $(A,D(A))$ iff the closure of $(A,D(A))$ is equal to the closure of operator $(A,H)$.
\end{definition}

\begin{remark}\label{remcore}
We recall (proposition 1.8 from \cite{EN}) that if $L$ is the generator of a $C_0$-semigroup on Banach space $\mathcal{F}$, then $\bigcap_{n=1}^\infty D(L^n)$ is dense in $\mathcal{F}$ and is a core for $L$. This implies that $D(L^n)$ is also a core of $L$ and is dense in $\mathcal{F}$ for all $n=1,2,3,\ldots$
\end{remark}

Now we are ready to state the Chernoff's theorem. From several options (see \cite{EN, Chernoff, BS, JL}), we choose the one given in~\cite{BS} (in equivalent formulation):
\begin{theorem}[\textsc{P.\,R.~Chernoff (1968)}, cf.~\cite{EN, Chernoff, BS, JL}]
\label{ChernoffTheor}
Suppose that the following three conditions are met:
\begin{enumerate}
\item 
$C_0$-semigroup $(e^{tL})_{t\ge 0}$ with generator $(L,D(L))$ in Banach space $\mathcal{F}$ is given, 
such that for some $w\geq 0$ the inequality $\|e^{tL}\| \le e^{wt}$ holds 
for all $t\ge0$. 

\item 
There exists a strongly continuous mapping $S\colon[0,+\infty)\to\mathscr{L}(\mathcal{F})$ such that $S(0)=I$ and the inequality $\|S(t)\| \le e^{wt}$ holds for all $t\ge0$.

\item 
There exists a dense linear subspace $D\subset\mathcal{F}$ such that for all $f\in D$ there exists a limit
$S'(0)f := \lim_{t\to +0} (S(t)f-f)/t$.
Moreover, $S'(0)$ on $D$ has a closure that coincides with the generator $(L,D(L))$.
\end{enumerate}
Then the following statement holds:
\begin{enumerate}
\item[(C)] for every $f\in \mathcal{F}$, as $n\to\infty$ we have $S(t/n)^n f \to e^{tL}f$ locally uniformly with respect to $t\ge0$, i.e. for each $T>0$ and each $f\in \mathcal{F}$ we have
$$
\lim_{n\to\infty}\sup_{t\in[0,T]}\|S(t/n)^n f - e^{tL}f\| = 0.
$$
\end{enumerate}
\end{theorem}

\begin{definition} \label{defChFun} 
Let $C_0$-semigroup $(e^{tL})_{t\ge 0}$ with generator $L$ in Banach space $\mathcal{F}$ be given. 
The mapping $S\colon [0,+\infty)\to \mathscr{L}(\mathcal{F})$ is called a \textit{Chernoff function for operator $L$} iff it satisfies the
condition (C) of Chernoff theorem~\ref{ChernoffTheor}.
In this case expressions $S(t/n)^n$ are called \emph{Chernoff approximations to the semigroup $e^{tL}$}.
\end{definition}

{\bf One-dimensional real analog of Chernoff's theorem.} In this subsection we discuss how the Chernoff theorem~\ref{ChernoffTheor} can be understood if Banach space $\mathcal{F}$ is one-dimensional, i.e. $\mathcal{F}=\mathbb{R}$. In this case any linear operator $A\in\mathscr{L}(\mathcal{F})$ is a multiplication by some real number $a$, any $C_0$-semigroup $(e^{tA})_{t\ge 0}$ consists of multiplications by numbers $e^{ta}$, i.e. $(e^{tA})f = e^{ta}\cdot f$ for any $t\ge0$ and any $f\in\mathcal{F}=\mathbb{R}$. Thus, theorem~\ref{ChernoffTheor} can be reformulated as follows:
\begin{theorem}\label{ChernoffTheor1dim}
Suppose there exists a function $s\colon[0,+\infty)\to\mathbb{R}$ such that $s(0)=1$ and the following conditions are met:
\begin{enumerate}
\item 
function $s(t)$ is continuous and for some $w\ge0$ the inequality $|s(t)|\le e^{wt}$ holds for all $t\ge0$;
\item 
there exists a right-side derivative 
$a = s'(0) := \lim_{t\to +0} (s(t)-1)/t$.
\end{enumerate}
Then $s(t/n)^n \to e^{ta}$ as $n\to\infty$ locally uniformly with respect to $t\ge0$, i.e. for each $T>0$ we have
$$
\lim_{n\to\infty}\sup_{t\in[0,T]}|s(t/n)^n - e^{ta}| = 0.
$$
\end{theorem}
Using the formula $\lim_{n\to\infty}(1+ta/n)^n=e^{ta}$, which is a statement from simple calculus, we can see that the first condition in the theorem~\ref{ChernoffTheor1dim} is redundant.
This is how we get the following short statement, which we call \emph{one-dimensional real analog of Chernoff's theorem}:
\begin{multline}\label{firsteq}
\Big(s\colon[0,+\infty)\to\mathbb{R}, s(0)=1, s'(0)=a\Big)\\
\Longrightarrow s(t/n)^n=e^{ta}+o(1)\textrm{ as }n\to\infty\textrm{ for all }t\ge0.
\end{multline}

{\bf The idea of the main result of the paper.} Let us consider for $\mathcal{F}=\mathbb{R}$ and fixed $m=1,2,3,\dots$ a more profound version of~(\ref{firsteq}):
\begin{multline}\label{seqeq}
\Big(s\colon[0,+\infty)\to\mathbb{R}, s(t) = \sum_{k=0}^m \frac{a^kt^k}{k!} + o(t^m)\textrm{ as }t\to0\Big) \\\Longrightarrow s(t/n)^n=e^{ta}+o(1/n^{m-1}) \textrm{ as }n\to\infty\textrm{ for all }t\ge0. \end{multline}
Statement (\ref{seqeq}) is similar to (\ref{firsteq}), but (\ref{seqeq}) is not so elementary even in one-dimensional case. The idea of (\ref{seqeq}) is the following: if one wants to approximate the exponent $e^{ta}$ using the fomula $e^{ta}=\lim_{n\to\infty}s(t/n)^n$, then the highest speed of convergence will be achieved if $s(t)=e^{ta}$. So good functions $s(t)$ should be close to the exponent $e^{ta}$ in some sense. In what sense? The answer is: in the infinitesimal sense, when $s(t)$ and $e^{ta}$ have the same Taylor polynomial. Higher  degree $m$ of the polynomial should provide higher speed $o(1/n^{m-1})$ of approximation. Statement~(\ref{seqeq}) is a one-dimensional version of our main theorem~\ref{mainth} and contains its main idea. Of course, the theorem~\ref{mainth} covers non-trivial cases, such as  $\mathrm{dim}\mathcal{F}=\infty$ and $\|L\|=\infty$.

{\bf Semigroups and linear evolution equations.} 
It is a well known fact~\cite{EN} that the solution of a well-posed Cauchy problem for a linear evolution partial differential equation (such as: Sch\"odinger-type equations, heat equation, parabolic equations) is given by a strongly continuous semigroup of linear bounded operators whose infinitesimal generator is a (usually unbounded) linear operator from the right-hand side of the evolution equation. Let us explain this in more detail. Let $X$ be an infinite set, and $\mathcal{F}$ be a Banach space of (not necessarily all) number-valued functions on $X$, and let $L$ be a closed linear operator $L\colon D(L)\to\mathcal{F}$ with the domain $D(L)\subset\mathcal{F}$ dense in $\mathcal{F}$. We consider the Cauchy problem for the evolution equation
\begin{equation}\label{CauchyP}
\left\{ \begin{array}{ll}
u'_t(t,x)=Lu(t,x),\\
u(0,x)=u_0(x),
\end{array} \right.
\end{equation}
where $x\in X$, $u_0\in\mathcal{F}$, $u(t,\cdot)\in \mathcal{F}$ for all $t\geq 0$.
Operator $L$ can be, in a trivial case, the Laplace operator $\Delta$ (so $u'_t=Lu$ is the heat equation), or (in a nontrivial case) a more sophisticated linear differential operator with variable coefficients that do not depend on $t$ but depend (usually nonlinearly) on $x$. It is known \cite{EN} that, in case the $C_0$-semigroup $\left(e^{tL}\right)_{t\geq 0}$ exists and has the generator $(L, D(L))$, the solution to Cauchy problem (\ref{CauchyP}) exists and is given by the equality $u(t,x)=(e^{tL}u_0)(x)$ for all $t\geq 0$ and $x\in X$. If  $u_0\in D(L)$, then $u(t,\cdot)\in D(L)$ for all $t\geq 0$ and the solution $u$ is a classical solution (in the terminology of \cite{EN}). And for arbitrary $u_0\in\mathcal{F}$ the solution of Cauchy problem~\eqref{CauchyP} exists  as a mild solution (in the terminology of \cite{EN}), i.e. the solution of the corresponding integral equation $u(t,\cdot)=L\int_0^tu(s,\cdot)ds + u_0$.

The equality $u(t,x)=(e^{tL}u_0)(x)$ for the solution of the Cauchy problem~(\ref{CauchyP}) shows that finding the semigroup  $\left(e^{tL}\right)_{t\geq 0}$ is a hard problem because it is equivalent to solving the Cauchy problem (\ref{CauchyP}) for each $u_0\in\mathcal{F}$. However, if a Chernoff function $S$ for operator $L$ is constructed (see definition~\ref{defChFun}), then the semigroup is given by the equality $e^{tL}=\lim_{n\to\infty}S(t/n)^n$. An advantage of this approach to solving~\eqref{CauchyP} arises from the fact that usually it is possible to define $S$ by an explicit and not very long formula which contains coefficients of operator $L$. This gives approximations to the solution of the Cauchy problem (\ref{CauchyP}) converging towards the solution in $\mathcal{F}$ as $n\to\infty$. Expressions $S(t/n)^nu_0$ are called \emph{Chernoff approximations} to the solution of the Cauchy problem~(\ref{CauchyP}).

{\bf What is new compared with the best known results in the field.} 
The Chernoff theorem has a long list of applications, but a short list of generalizations and developments because it is difficult to obtain them. Original Chernoff's proof and its variants given in all textbooks known to us are difficult to generalize. To our current knowledge all contributions to ''theory of rates of convergence in Chernoff's theorem'' can be found in \cite{GSK2019, Zag2020} and references therein. There are also few ``practical'' research papers \cite{OSS2012, Prud} that measure the speed of convergence in particular cases obtained via numerical simulations. In the present paper we propose a completely new approach that allows for simpler proofs and more general results. Let us note that if $S(t)$ is a Chernoff function for operator $L$, then the speed of convergence of $S(t/n)^nf$ to $e^{tL}f$ depends both on $S(t)$ and $f\in\mathcal{F}$, even if $\|f\|=1$. Not all Chernoff functions $S(t)$ and vectors $f$ provide high speed of approximation as our examples show (see section~\ref{ExOfASAFC}). This is a very important and commonly not noticed fact: for example in~\cite{Zag2020} estimates in norm operator topology in the space $\mathscr{L}(\mathcal{F})$ are considered hence dependence on direction of $f$ is out of the scope of~\cite{Zag2020}, meanwhile \cite{Zag2020} is probably one of the best recent papers on the topic. In another bright paper~\cite{GSK2019} dependence on $f$ is taken into account but our theorem~\ref{mainth} is much more general because it works for arbitrary $k=1,2,3,\dots$ in~\eqref{seqeq} and has the form that is very suitable for practical use. The present paper is a continuation of our research \cite{VVGKR, GR-2021}.

{\bf Applications.}  In the last section of the paper we provide (with full proofs) an example of application of theorem~\ref{mainth}, which itself is helpful but can also be used as a template for further applications. As far as we know this is the first example of a rigorous estimation of the speed of convergence for Chernoff approximations for solution to the Cauchy problem for a concrete class of equations (second order parabolic equations with variable coefficients), see theorem~\ref{teorApprDU2} and  example~\ref{exremizf}. In \cite{Butko-2019} one can find many classes of equations for which solution methods based on the Chernoff approximations are developed, so we expect many cases for application of our main theorem \ref{mainth}.
One very simple example (rapidly converging Chernoff approximations for solutions to the heat equation) can be found in \cite{VVGKR}. 

{\bf Finally, what this paper is about.}  This paper is devoted to the study of the speed of vanishing of the norm of the difference between semigroup $e^{tL}$ and its Chernoff approximation $S(t/n)^n$. We estimate $\|e^{tL}f-S(t/n)^nf\|$ for a fixed $f\in\mathcal{F}$ and all large enough $n$. The main result of the paper is the theorem~\ref{mainth}. When we say that arbitrary Banach space $\mathcal{F}$ is given, we assume it to be over fields $\mathbb{R}$ or $\mathbb{C}$, all the statements in this setting are true for both cases.

\section{Examples of arbitrary slow and arbitrary fast convergence}
\label{ExOfASAFC}

Let us first provide examples of arbitrary fast and arbitrary slow convergence. We proposed our first examples of such kind in \cite{VVGKR}, and now we develop them. 

The following fact should be well known, but a clear short proof is better than a reference. The $C_0$-(semi)group of translations will be basic for (counter)examples provided in this section. 

\begin{lemma}[On the group of translations]\label{translemma} 
Consider the linear space $\mathcal{F}=UC_b(\mathbb{R})$ of all uniformly continuous bounded functions $f\colon\mathbb{R}\to\mathbb{R}$ with the uniform norm $\|f\|=\sup_{x\in\mathbb{R}}|f(x)|$ which makes $UC_b(\mathbb{R})$ a Banach space. 
Define $(Q(t)f)(x)=f(x+t)$ for all $t,x\in\mathbb{R}$ and all $f\in UC_b(\mathbb{R})$.
Then:

\hangindent=3em
1. $(Q(t))_{t\in\mathbb{R}}$ is a $C_0$-group in $UC_b(\mathbb{R})$.


\hangindent=3em
2. The generator $(L,D(L))$ of the $C_0$-group $(Q(t))_{t\in\mathbb{R}}$ is given by $L=[f\mapsto f']$, i.e.  $(Lf)(x)=f'(x)$ and $$D(L)=UC_b^1(\mathbb{R})\stackrel{define}{=}\{f|f,f'\in UC_b(\mathbb{R})\}.$$ 
From now let us use notation $Q(t)=e^{tL}$.

\hangindent=3em
3. $(Q(t))_{t\ge0}$ is a $C_0$-semigroup in $UC_b(\mathbb{R})$ with the same generator $(L,D(L))$.

\hangindent=3em
4. $D(L^n)=UC_b^n(\mathbb{R})\stackrel{define}{=}\{f|f,f',\dots,f^{(n)}\in UC_b(\mathbb{R})\}.$

\hangindent=3em
5. Operator $(f\mapsto f', UC_b^1(\mathbb{R}))$ is closed in $UC_b(\mathbb{R})$.

\hangindent=3em
6. Each of the spaces $UC_b^n(\mathbb{R})$ is dense in $UC_b(\mathbb{R})$ and is a core for $(f\mapsto f', UC_b^1(\mathbb{R}))$.

\hangindent=3em
7. $\|e^{tL}\|=1$ for all $t\in\mathbb{R}$.
\end{lemma}
\begin{proof}
1. Conditions $Q(0)f=f$ and $Q(t_1)Q(t_2)f=Q(t_1+t_2)f$ follow directly from the formula $(Q(t)f)(x)=f(x+t)$. Condition $\lim_{t\to 0}\|Q(t)f-f\|=0$ for each $f\in UC_b(\mathbb{R})$ follows from the fact that $f$ is uniformly continuous. Indeed, for a fixed $f\in UC_b(\mathbb{R})$ and $\varepsilon>0$ there exists $\delta>0$ such that inequality $|t|<\delta$ implies $|f(x+t)-f(x)|<\varepsilon$ for all $x\in\mathbb{R}$, so $\|Q(t)f-f\|=\sup_{x\in\mathbb{R}}|f(x+t)-f(x)|\leq\varepsilon$ for all $|t|<\delta$. 

2a. Let us prove that $D(L)\subset UC_b^1(\mathbb{R})$. Suppose $f\in D(L)\subset UC_b(\mathbb{R})$, then by definition~\ref{defgen} of a generator we have 
$0=\lim_{t\to 0}\|(e^{tL}f-f)/t-Lf\|=\lim_{t\to 0}\sup_{x\in\mathbb{R}}|(f(x+t)-f(x))/t-(Lf)(x)|$, so $\frac{1}{t}(f(x+t)-f(x))\to (Lf)(x)$ uniformly (hence pointwise) as $t\to 0$. Pointwise convergence implies that at each $x\in\mathbb{R}$ function $f$ is differentiable and $f'(x)=(Lf)(x)$. So $f'=Lf\in UC_b(\mathbb{R})$. We have thus proved that $f,f'\in UC_b(\mathbb{R})$, hence $f\in UC_b^1(\mathbb{R})$.

2b. Let us now prove that $UC_b^1(\mathbb{R})\subset D(L)$. To do that we need to take $f\in UC_b^1(\mathbb{R})$ and prove that $\frac{1}{t}(f(x+t)-f(x))\to f'(x)$ uniformly in $x\in\mathbb{R}$ as $t\to 0$. Let us prove by contradiction: suppose there exists such $\varepsilon_0>0$ that for each $\delta>0$ there exist such $t_\delta\in(-\delta,\delta)$ and such $x_\delta\in\mathbb{R}$ that $|\frac{1}{t_\delta}(f(x_\delta+t_\delta)-f(x_\delta)) - f'(x_\delta)|\geq \varepsilon_0$. As $f'$ exists, by Lagrange's theorem there exists $\xi_\delta\in(x_\delta,x_\delta+t_\delta)$ such that $\frac{1}{t_\delta}(f(x_\delta+t_\delta)-f(x_\delta))=f'(\xi_\delta)$, so $|f'(\xi_\delta) - f'(x_\delta)|\geq \varepsilon_0$. But $|\xi_\delta-x_\delta|<|t_\delta|<\delta$ and $\delta>0$ is arbitrary so we have a contradiction with the fact that $f'$ is uniformly continuous. Then $f\in D(L)$.

3. See remark after definition of a generator of $C_0$-group in~\cite[p.~79]{EN}.

4. Directly follows from item 2 and definition~\ref{remdom}.

5. Assume that $f_n\in UC_b^1(\mathbb{R})$ for all $n=1,2,3,\ldots$ and $f_n\to f$, assume that there exists $g\in UC_b(\mathbb{R})$ such that $f'_n\to g$. We need to prove that $f\in UC_b^1(\mathbb{R})$ and $f'=g$. This all follows from theorems of calculus on differentiation under the limit sign. Indeed, if $f_n$ converges to $f$ uniformly, and $f_n'$ converges to $g$ uniformly then $g$ is differentiable and $f'=g$. The condition $f\in UC_b^1(\mathbb{R})$ follows from the fact that $f\in UC_b(\mathbb{R})$ and $f'=g\in UC_b(\mathbb{R})$. 

6. Consider the space $C_b^\infty(\mathbb{R})$ of functions bounded with all derivatives.
Then $C_b^\infty(\mathbb{R})$ is dense in $UC_b(\mathbb{R})$ due to lemma 1 in \cite{RemJMP}. Also we have $C_b^\infty(\mathbb{R})\subset UC_b^n(\mathbb{R})$.
So $UC_b^n(\mathbb{R})$ is dense in $UC_b(\mathbb{R})$.
The fact that the closure of $(f\mapsto f', UC_b^n(\mathbb{R}))$ is $(f\mapsto f', UC_b^1(\mathbb{R}))$ is shown exactly by the reasoning that we used in the proof of item 5. 

7. We have $\|e^{tL}f\|=\sup_{x\in\mathbb{R}}|f(x+t)|=\sup_{x\in\mathbb{R}}|f(x)|=\|f\|$ so $\|e^{tL}\|=1$ for all $t\in\mathbb{R}$.

\end{proof}

Later we will work with the notion of modulus of continuity. For fixing notation and details we recall (following \cite{Dz}, pp. 168-174) the definition and some simple facts concerning this notion. 

\begin{definition}
For a uniformly continuous function $f\colon \mathbb{R}\to\mathbb{R}$ its modulus of continuity is a function $\omega_f\colon [0,+\infty)\to [0,+\infty)$ defined by the equality
$$
\omega_f(x)=\sup_{|x_1-x_2|\leq x}|f(x_1)-f(x_2)|.
$$
\end{definition}

\begin{proposition}\label{prop4cond}
1. Function $m\colon [0,+\infty)\to [0,+\infty)$ is a modulus of continuity for some uniformly continuous function $f\colon \mathbb{R}\to\mathbb{R}$ iff the following conditions i)-iv) hold:

i) $m(0)=0$;

ii) $m$ is non-decreasing: $x_1>x_2$ implies $m(x_1)\geq m(x_2)$;

iii) $m$ is continuous;

iv) $m$ is semiadditive in the sense that for all $x_1\geq 0$, $x_2\geq 0$ we have $m(x_1+x_2)\leq m(x_1)+m(x_2)$.

2. If i)-iv) hold, then $m$ is a modulus of continuity for itself, i.e. if we set $f(x)=m(x)$ for $x\geq 0$ and $f(x)=0$ for $x<0$, then $\omega_f(x)=m(x)$ for all $x\geq 0$.
\end{proposition}

\begin{proposition}\label{propnincr}
If $m\colon [0,+\infty)\to [0,+\infty)$ and function $x\longmapsto \frac{m(x)}{x}$ is non-increasing for $x>0$, then $m$ is semiadditive, i.e. condition iv) of the proposition \ref{prop4cond} holds.
\end{proposition}

\begin{remark}
If for function $f\colon\mathbb{R}\to\mathbb{R}$ we have $\omega_f(h)=o(h)$ as $h\to+0$, then $f(x)\equiv\mathrm{const}$. Indeed, for each $x\in\mathbb{R}$ and $h\neq 0$ we have 
$$
0\leq |f(x+h)-f(x)|\leq \sup_{|x_1-x_2|\leq|h|}|f(x_1)-f(x_2)|=\omega_f(|h|),$$
$$
0\leq \frac{|f(x+h)-f(x)|}{|h|}\leq \frac{\omega_f(|h|)}{|h|},
$$
so $\lim\limits_{h\to 0}\left|\frac{f(x+h)-f(x)}{h}\right|=0$ hence $\lim\limits_{h\to 0}\frac{f(x+h)-f(x)}{h}=0$ i.e.  $f'(x)= 0$ for each $x\in\mathbb{R}$  hence $f(x)\equiv\mathrm{const}$. This is why for a non-constant uniformly continuous function $f$ such cases as $\omega_f(h)=\sqrt{h}$ and $\omega_f(h)=2h$ are possible but such case as  $\omega_f(h)=2h\sqrt{h}$ is not possible.
\end{remark}

Now let us provide a family of Chernoff functions for the (semi)group of translations. Function $v$ serves as a parameter in this family and determines the speed of convergence of the Chernoff approximations.

\begin{theorem}\label{thlow1}
Consider $\mathcal{F}=UC_b(\mathbb{R})$ -- the space of all uniformly continuous bounded functions $f\colon\mathbb{R}\to\mathbb{R}$ with the uniform norm $\|f\|=\sup_{x\in\mathbb{R}}|f(x)|$. 
Consider the group of translations $(e^{tL}f)(x)=f(x+t)$ in $UC_b(\mathbb{R})$ described in lemma \ref{translemma}. Suppose that  function $v\colon (0,+\infty)\to[0,+\infty)$ satisfies the condition $\lim_{x\to+\infty}v(x)=0$. For each $f\in UC_b(\mathbb{R})$ define $G(0)f=f$ and
\begin{equation}\label{basicChfuncdef}
(G(t)f)(x)=f(x+t+tv(1/t))\textrm{ for all }x\in\mathbb{R}, t>0. 
\end{equation}
Then: 1. For all $t\geq 0$ we have $\|e^{tL}\|=\|G(t)\|=1$.

2. $G$ is a Chernoff function for $(e^{tL})_{t\geq 0}$, i.e. for all $T>0$ we have 
$$
\lim_{n\to\infty}\sup_{t\in[0,T]}\|G(t/n)^n f-e^{tL}f\|=0\textrm{ for each }f\in UC_b(\mathbb{R}).
$$

3. If, additionally, function $v$ is continuous and non-increasing everywhere on $(0,+\infty)$, then for all $f\in UC_b(\mathbb{R})$ and all $T>0$ we have
\begin{equation}
\sup_{t\in[0,T]}\|G(t/n)^n f-e^{tL}f\|=\omega_f(Tv(n/T))\textrm{ for each }n=1,2,3,\dots
\end{equation}
where $\omega_f$ is the modulus of continuity of the function $f$.
\end{theorem}

\begin{proof}Before checking the proof please see lemma~\ref{translemma} for the properties of the (semi)group of translations.

1. Item 7 of lemma \ref{translemma} states that $\|e^{tL}\|=1$. It is clear that $G(t)$ is a linear bounded operator for each $t\geq 0$. For fixed $t\geq 0$ we have $\|G(t)f\|=\sup_{x\in\mathbb{R}}|f(x+t+tv(1/t))|=\sup_{y\in\mathbb{R}}|f(y)|=\|f\|$ so $\|G(t)\|=1$ for all $t\geq 0$.

2. It follows from the definition (\ref{basicChfuncdef}) of function $G$ that  $(G(t/n)f)(x)=f(x+t/n+(t/n)v(n/t))$ and $(G(t/n)^nf)(x)=f(x+t+tv(n/t))$, so 
$$
\sup_{t\in[0,T]}\|e^{tL}f-G(t/n)^nf\|=
\textrm{ [due to $e^{tL}f-G(t/n)^nf=0$ as $t=0$] }=
$$
$$
\sup_{t\in(0,T]}\|e^{tL}f-G(t/n)^nf\|=
\sup_{t\in(0,T]}\sup_{x\in\mathbb{R}}|f(x+t)-f(x+t+tv(n/t))|=
$$
$$
=\textrm{ [change of variable $x+t=y$] }=
\sup_{t\in(0,T]}\sup_{y\in\mathbb{R}}|f(y)-f(y+tv(n/t)).
$$
Then by changing the order of supremums we obtain:
\begin{equation}\label{eqProof2}
\sup_{t\in[0,T]}\|e^{tL}f-G(t/n)^nf\| = \sup_{y\in\mathbb{R}}\sup_{t\in(0,T]}|f(y)-f(y+tv(n/t))|.
\end{equation}
Function $f\in UC_b(\mathbb{R})$ is uniformly continuous so
for each $\varepsilon>0$ there exists such $\delta>0$
that for each $y\in\mathbb{R}$ condition $tv(n/t)<\delta$ implies inequality $|f(y)-f(y+tv(n/t))|<\varepsilon$. We have $n/t \ge n/T$ for all $t\in(0,T]$, and $\lim_{z\to+\infty}v(z)=0$.
So if $z_0$ is large enough to guarantee that $Tv(z)<\delta$ for all $z>z_0$, then for all $n>Tz_0$ and all $t\in(0,T]$ we have $tv(n/t)\le Tv(n/t)<\delta$,
which implies $|f(y)-f(y+tv(n/t))|<\varepsilon$ for all $n>Tz_0$, all $t\in(0,T]$ and all $y\in\mathbb{R}$.
Hence we get
$$
\lim_{n\to\infty}\sup_{y\in\mathbb{R}}\sup_{t\in(0,T]}\big|f(y)-f(y+tv(n/t))\big| = 0.
$$
This proves item 2 thanks to  equality~\eqref{eqProof2}.

3. Let us use the equality~\eqref{eqProof2} once more. Thanks to conditions in the item 3 of the theorem function $(0,+\infty)\ni x\longmapsto v(x)$ is non-increasing and continuous. So the function $(0,T]\ni t\mapsto tv(n/t)$ is non-decreasing and  continuous hence it maps the interval $(0,T]$ onto the interval $(0,Tv(n/T)]$. Performing a change of variable $\tau=tv(n/t)$ in~\eqref{eqProof2} we get:
$$
\sup_{t\in[0,T]}\|e^{tL}f-G(t/n)^nf\| = \sup_{y\in\mathbb{R}}\sup_{0<\tau\le Tv(n/T)}|f(y)-f(y+\tau)|=
$$
$$
=[x=y+\tau,\tau=x-y]=\sup_{0<x-y\leq Tv(n/T)}|f(y)-f(x)|=\omega_f(Tv(n/T)),
$$
where $\omega_f$ is the modulus of continuity of function $f$. Recall that $\omega_f$ is well-defined because $f$ is uniformly continuous. Item 3 is proved.
\end{proof}

With the above theorem we can provide examples powerful enough to answer rather general questions. The following proposition gives an example of Chernoff approximations that converge on each vector but do not converge in operator norm.

\begin{proposition} There exists a
Banach space $\mathcal{F}$, $C_0$-semigroup $(e^{tL})_{t\geq 0}$ in $\mathcal{F}$ with generator $(L,D(L))$, and Chernoff function $G$ for operator $(L,D(L))$ such that: 

1. $\lim_{n\to\infty}\|G(t/n)^nf-e^{tL}f\|=0$ for all $f\in\mathcal{F}$,

2. $\|e^{tL}\|=\|G(t)\|=1$,

3. for each $t>0$ and each $n\in\mathbb{N}$ there exists $f_n\in\mathcal{F}$ such that $\|f_n\|=1$ and $\|G(t/n)^nf_n-e^{tL}f_n\|\geq \|f_n\|$ so $\|G(t/n)^n-e^{tL}\|\geq 1 \not\to 0$ as $n\to\infty$.
\end{proposition}

\begin{proof}
Indeed, consider $\mathcal{F}=UC_b(\mathbb{R})$, $(e^{tL}f)(x)=f(x+t)$, and set $v(t)=1/t$ in theorem \ref{thlow1}, then $(G(t)f)(x)=f(x+t+tv(1/t))$ becomes $(G(t)f)(x)=f(x+t+t^2)$ and item 2 of theorem \ref{thlow1} says that $\lim_{n\to\infty}\|G(t/n)^nf-e^{tL}f\|=0$ for all $f\in\mathcal{F}$. Item 1 is proved. Item 2 holds due to item 1 of theorem \ref{thlow1}. 

Let us prove item 3. Suppose that $t>0$ is fixed and define
$$
f_n(x)=\left\{ \begin{array}{ll}
	0  & \textrm{ for }\ x\leq 0,\\
	\frac{n}{t^2}x &  \textrm{ for }\  0<x<t^2/n,\\
	1\ & \textrm{ for }\   x\geq t^2/n.
	\end{array}\right.
$$
Then
$$
(e^{tL}f_n)(x)=f_n(x+t)=\left\{ \begin{array}{ll}
	0  & \textrm{ for }\ x+t\leq 0,\\
	\frac{n}{t^2}(x+t) &  \textrm{ for }\  0<x+t<t^2/n,\\
	1\ & \textrm{ for }\   x+t\geq t^2/n.
	\end{array}\right.
$$
$$
(e^{tL}f_n)(x)=f_n(x+t)=\left\{ \begin{array}{ll}
	0  & \textrm{ for }\ x\leq -t,\\
	\frac{n}{t^2}(x+t) &  \textrm{ for }\  -t<x<t^2/n-t,\\
	1\ & \textrm{ for }\   x\geq t^2/n-t.
	\end{array}\right.
$$
It directly follows from $(G(t)f)(x)=f(x+t+t^2)$ that $(G(t/n)f)(x)=f(x+t/n+(t/n)^2)$ and $(G(t/n)^nf)(x)=f(x+t+t^2/n)$ for all $f\in\mathcal{F}$. So
$$
(G(t/n)^nf_n)(x)=\left\{ \begin{array}{ll}
	0  & \textrm{ for }\ x\leq -t-t^2/n,\\
	\frac{n}{t^2}(x+t+t^2/n) &  \textrm{ for }\  -t-t^2/n<x<t^2/n-t-t^2/n,\\
	1\ & \textrm{ for }\   x\geq t^2/n-t-t^2/n.
	\end{array}\right.
$$
$$
(G(t/n)^nf_n)(x)=\left\{ \begin{array}{ll}
	0  & \textrm{ for }\ x\leq -t-t^2/n,\\
	\frac{n}{t^2}x+n/t+1 &  \textrm{ for }\  -t-t^2/n<x<-t,\\
	1\ & \textrm{ for }\   x\geq -t.
	\end{array}\right.
$$
Then for $x_t=-t$ we have $(e^{tL}f_n)(x_t)=0$ and $(G(t/n)^nf_n)(x_t)=1$, so $\|e^{tL}f_n-G(t/n)^nf_n\|=\sup_{x\in\mathbb{R}}|(e^{tL}f_n)(x)-(G(t/n)^nf_n)(x)|\geq |0-1|=1$. Item 3 is proved.
\end{proof}

\begin{proposition}\label{prop1}
For an arbitrary non-increasing continuous function $v\colon (0,+\infty)\to [0,+\infty)$ vanishing at infinity at arbitrary high rate (e.g. $v(x)=(1+x)^{-k}$, $v(x)=e^{-x}$, $v(x)=e^{-e^x}$) there exist $C_0$-semigroup $(e^{tL})_{t\geq 0}$ with generator $(L,D(L))$ in Banach space $\mathcal{F}$, Chernoff function $G$ and vector $f\in\mathcal{F}$ such that $f\notin D(L)$ but the speed of convergence of Chernoff approximations $G(t/n)^n f$ is arbitrary high, i.e. for all $T>0$ we have $\sup_{t\in[0,T]}\|G(t/n)^n f-e^{tL}f\|=Tv(n/T)$ for all $n=1,2,3,\dots$ such that $Tv(n/T)\leq 1$. Moreover, we have $\|e^{tL}\|=\|G(t)\|=\|f\|=1$ for all $t\ge0$.
\end{proposition}

\begin{proof}
Indeed, set $\mathcal{F}$, $(e^{tL})_{t\geq 0}$, $G$ as in theorem \ref{thlow1} and define $f(x)=\max(0,\min(x,1))$ for all $x\in\mathbb{R}$. Then $f$ is not differentiable at $0$ 
so $f\notin D(L)$, and $\omega_f(x)=x$ for $x\in[0,1]$, hence proposition is correct thanks to item 3 of theorem~\ref{thlow1}.
\end{proof}

\begin{remark}
It is possible to show that after a slight modification of proposition~\ref{prop1} there exists not only $f\notin D(L)$ on which the speed of convergence is arbitrary high, but also vector $g\in\mathcal{F}$ on which the convergence is arbitrary slow.
\end{remark}

\begin{proposition}
There exist $C_0$-semigroup $(e^{tL})_{t\geq 0}$ in Banach space $\mathcal{F}$, Chernoff function $G$ and vector $f\in\mathcal{F}$ such that $f\in \cap_{j=1}^\infty D(L^j)$ but the speed of convergence is arbitrary low, i.e. for arbitrary chosen non-increasing continuous function $u\colon (0,+\infty)\to [0,+\infty)$ vanishing at infinity at arbitrary low rate (e.g. $u(x)=(1+x)^{-1/k}$, $u(x)=1/\ln(x+e)$, $u(x)=1/\ln(\ln(x+e^e))$) and all $T>0$ we have $\sup_{t\in[0,T]}\|G(t/n)^n f-e^{tL}f\|=Tu(n/T)$ for all $n=1,2,3,\dots$ such that $Tu(n/T)\leq 1$. Moreover, we have $\|e^{tL}\|=\|G(t)\|=\|f\|=1$ for all $t\ge0$.
\end{proposition}

\begin{proof}
Indeed, set $\mathcal{F}$, $(e^{tL})_{t\geq 0}$, $G$ as in theorem \ref{thlow1}, $v=u$ and define
$$
f(x)=\left\{ \begin{array}{ll}
	2&  \textrm{ for }\  x\geq3,\\
	x  & \textrm{ for }\ x\in[0,1],\\
	C^\infty\textrm{-continued with }0\leq f'(x)\leq 1\ & \textrm{ for }\   x\in [1,3],
	\\
	-f(-x)&  \textrm{ for }\  x<0.
\end{array}\right.
$$
Each derivative of function $f$ is continuous on $\mathbb{R}$ and vanishes outside $[-3,3]$, and so it is bounded; hence $f\in \cap_{j=1}^\infty D(L^j)$. Also $\omega_f(x)=x$ for $x\in[0,1]$; hence the proposition is correct thanks to item 3 of the theorem \ref{thlow1}.
\end{proof}

\begin{remark}
The examples above given show that even in a very natural and simple setting we should not expect the convergence in the operator norm $\|S(t/n)^n-e^{tL}\|\to 0$ as $n\to\infty$. Instead, in general setting (i.e. under conditions of the Chernoff theorem~\ref{ChernoffTheor}) we have only convergence $\|S(t/n)^nf-e^{tL}f\|\to 0$ as $n\to\infty$ on every vector $f$, and this convergence may be arbitrary slow, so we need to choose the vector $f$ and the Chernoff function $S$ wisely if we want to have fast convergence. 

In the next section we provide conditions that guarantee high speed of convergence of Chernoff approximations. Under some conditions we show that if $S(0)=I$, $S'(0)=L$, $S''(0)=L^2$,\ldots,$S^{(m)}(0)=L^m$ and the difference $S(t)f-\sum_{k=0}^mt^kL^k/k!f$ is estimated properly on a suitable set of vectors $f$, then 
$\|S(t/n)^nf-e^{tL}f\|$ behave as $1/n^m$ or close to it, depending on how we estimate the difference $S(t)f-\sum_{k=0}^mt^kL^k/k!f$. 
\end{remark}

\section{Estimates for fast convergence (main result)}

We start from the simple, purely algebraic lemma that establishes the decomposition that is basic for our approach.

\begin{lemma}
Let $Z$ and $Y$ be elements of a ring with associative (but maybe
non-commutative) multiplication with unity (e.g. $Z$ and $Y$ may be linear, everywhere defined operators mapping some linear space into itself). Then the following equality holds:
\begin{equation}\label{algebr}
    Z^n - Y^n =\sum_{k=0}^{n-1}Z^{n-k-1}(Z-Y)Y^k.
\end{equation}
Proof.
\end{lemma}
$$R.h.s.=\sum_{k=0}^{n-1}Z^{n-k-1}(Z-Y)Y^k = \sum_{k=0}^{n-1}Z^{n-k}Y^k - \sum_{k=0}^{n-1}Z^{n-k-1}Y^{k+1}=
$$
$$
=\left(Z^nY^0+\sum_{k=1}^{n-1}Z^{n-k}Y^k\right)-\left(\sum_{j=0}^{n-2}Z^{n-j-1}Y^{j+1}+Z^{n-(n-1)-1}Y^{n-1+1}\right)\stackrel{j=k-1}{=}
$$
$$
=Z^n + \sum_{k=1}^{n-1}Z^{n-k}Y^k-\sum_{k=1}^{n-1}Z^{n-k}Y^k-Y^{n}=Z^n-Y^n=L.h.s. \eqno\square
$$

\medskip
This lemma has the following corollary regarding the high speed of convergence of Chernoff approximations.

\begin{lemma}\label{mainlemmaoc}
Suppose that the following three conditions are met:
\begin{enumerate}
\item 
$C_0$-semigroup $(e^{tL})_{t\ge 0}$ with generator $(L,D(L))$ in Banach space $\mathcal{F}$ is given, 
such that for some $M_1\geq 1$, $w\geq 0$ and $T>0$ the inequality $\|e^{tL}\| \le M_1e^{wt}$ holds 
for all $t\in [0,T]$. 

\item 
There exists a mapping $S\colon (0,T]\to\mathscr{L}(\mathcal{F})$ such that for some constant $M_2\geq 1$ the inequality $\|S(t)^k\| \le M_2e^{kwt}$ holds for all $t\in (0,T]$ and all $k=1,2,3,\dots$.

\item 
Numbers  $m\in\{0,1,2,\dots\}$ and $p\in\{1,2,3,\dots\}$ are fixed. There exists a $(e^{tL})_{t\ge 0}$-invariant subspace $\mathcal{D}\subset D(L^{m+p})\subset \mathcal{F}$ (i.e. $(e^{tL})(\mathcal{D})\subset \mathcal{D}$ for any $t\geq0$) and functions $C_j\colon(0,T]\to [0,+\infty)$, $j=0,1,\dots,m+p$ such that for all $t\in(0,T]$ and all $f\in\mathcal{D}$ we have 
\begin{equation} \label{ocdiff}
\left\|S(t) f - e^{tL}f\right\| \le t^{m+1}\sum_{j=0}^{m+p}C_j(t)\|L^{j}f\|.
\end{equation}
\end{enumerate}
Then for all $t>0$, all integer $n\geq t/T$ and all $f\in \mathcal{D}$ the following estimate is true:
\begin{equation} \label{ocresdiff1}
\|S(t/n)^n f - e^{tL}f\|\leq \frac{M_1M_2t^{m+1}e^{wt}}{n^{m}}\sum_{j=0}^{m+p}e^{-wt/n}C_j(t/n)\|L^{j}f\|.
\end{equation}
\end{lemma}

\begin{proof}
Setting $Z=S(t/n), Y=e^{(t/n)L}$ in formula (\ref{algebr}) we obtain
$$\|S(t/n)^n f - e^{tL}f\| \stackrel{by\;(\ref{algebr})}{=} 
\bigg\|\sum_{k=0}^{n-1} S(t/n)^{n-k-1} \left(S(t/n)-e^{(t/n) L}\right) \left(e^{(t/n) L}\right)^kf\bigg\|\leq
$$
$$
\leq \sum_{k=0}^{n-1} \left\|S(t/n)^{n-k-1}\right\| \cdot\left\|\left(S(t/n)-e^{(t/n) L}\right) \left(e^{(t/n) L}\right)^kf\right\|\stackrel{by\;(\ref{ocdiff})}{\leq}
$$
[here we put $\left(e^{(t/n) L}\right)^kf$ in the place of $f$ in (\ref{ocdiff})]
$$
\leq \sum_{k=0}^{n-1} \left\|S(t/n)^{n-k-1}\right\|\cdot\frac{t^{m+1}}{n^{m+1}}\sum_{j=0}^{m+p}C_j(t/n)\left\|L^{j}\left(e^{(t/n) L}\right)^kf\right\|=
$$
[here we use the fact that $C_0$-semigroup $(e^{tL})_{t\geq 0}$ maps $\mathcal{D}$ into $\mathcal{D}$ and commutes with $L^{j}$]
$$
=\sum_{k=0}^{n-1} \left\|S(t/n)^{n-k-1}\right\|\cdot\frac{t^{m+1}}{n^{m+1}}\sum_{j=0}^{m+p}C_j(t/n)\left\|\left(e^{(t/n) L}\right)^kL^{j}f\right\|\leq
$$
$$
\leq \sum_{k=0}^{n-1} \left\|S(t/n)^{n-k-1}\right\|\cdot\frac{t^{m+1}}{n^{m+1}}\|e^{(kt/n) L}\|\sum_{j=0}^{m+p}C_j(t/n)\|L^{j}f\|\leq
$$
$$
\leq \sum_{k=0}^{n-1} M_2e^{(n-k-1)wt/n}\cdot\frac{t^{m+1}}{n^{m+1}}M_1e^{w(kt/n)}\sum_{j=0}^{m+p}C_j(t/n)\|L^{j}f\|=
$$
$$
= \sum_{k=0}^{n-1}M_1M_2\frac{t^{m+1}}{n^{m+1}} e^{wt(n-1)/n}\sum_{j=0}^{m+p}C_j(t/n)\|L^{j}f\| =
$$
$$
=M_1M_2\frac{t^{m+1}}{n^{m}}e^{wt}\sum_{j=0}^{m+p}e^{-wt/n}C_j(t/n)\|L^{j}f\|.
$$
\end{proof}

Usually an a priori estimate in the form (\ref{ocdiff}) is not known. To overcome this problem we recall in the following lemma~\ref{lemmaTeylocexp} one fact which is most likely known but a short proof is better than reference. Using this fact, one can obtain (\ref{ocdiff}) studying only the norm of difference between $S(t)$ and its Taylor's polynomial because, as we will see now, $e^{tL}$ can also be approximated by (the same!) Taylor's polynomial.

\begin{lemma}\label{lemmaTeylocexp}
Let $\mathcal{F}$ be a Banach space, let $(e^{tL})_{t\ge 0}$ be a $C_0$-semigroup in $\mathcal{F}$ with generator $(L,D(L))$. 
Then for all $t\ge0$, all $m=0,1,2,\dots$ and all $f\in D(L^{m+1})$ we have the following formulas, where the integral is understood in Bochner's sense:
\begin{equation}\label{teylexp-f}
e^{tL}f = \sum_{k=0}^{m} \frac{t^kL^k f}{k!} + \int_0^t \frac{(t-s)^m}{m!}e^{sL}L^{m+1}f\, ds,
\end{equation}
\begin{equation}\label{teylexp-oc}
\left\|e^{tL}f - \sum_{k=0}^{m} \frac{t^kL^k f}{k!}\right\| \le \frac{t^{m+1}}{(m+1)!}\|L^{m+1}f\|\cdot\!\sup_{s\in[0,t]}\left\|e^{sL}\right\|.
\end{equation}
\end{lemma}
\begin{proof}
Denote $Q(t)=e^{tL}$. By definition of the generator of a $C_0$-semigroup (see definition \ref{generdef}), function $t\mapsto Q(t)f$ is differentiable at $t=0$ iff $f\in D(L)$, and $Q'(0)f=Lf$.
By the semigroup composition property (see definition \ref{semigrdef}) this implies the differentiability of this function at all $t\in[0,+\infty)$. 
The derivative at arbitrary time $t\geq 0$ can be found (using only the above-mentioned definitions) as follows:
\begin{multline*}
Q'(t)f=\lim_{h\to 0}\frac{1}{h}(Q(t+h)f-Q(t)f)=\lim_{h\to+0}\frac{1}{h}(Q(t)Q(h)f-Q(t)f)=\\
= Q(t)\lim_{h\to+0}\frac{1}{h}(Q(h)f-f) = Q(t)Q'(0)f = Q(t)Lf.
\end{multline*}
So the derivative is expressed in terms of the semigroup. Then for $f\in D(L)$ function $t\mapsto Q'(t)f=Q(t)Lf$ is differentiable at $t\in[0,+\infty)$ iff $Lf\in D(L)$ which is equivalent to $f\in D(L^2)$; moreover, $Q''(0)f=LLf=L^2f$. Repeating this argument we see that function $t\mapsto Q(t)f$ is $k$ times differentiable at $t\in[0,+\infty)$ iff $f\in D(L^k)$; for such $f$ we have $Q^{(k)}(t)f=Q(t)L^kf$. 

General Taylor's formula \cite[th. 12.4.4]{BS} after rescaling reads as
$$
F(t)=F(0)+F'(0)t+\dots+\frac{t^m}{m!}F^{(m)}(0)+\frac{1}{m!}\int_0^t(t-s)^mF^{(m+1)}(s)ds, 
$$
and for $F(t)=Q(t)f=e^{tL}f$, $F^{(k)}(t)f=e^{tL}L^kf$ becomes (\ref{teylexp-f}). 

Formula (\ref{teylexp-oc}) is a simple corollary of (\ref{teylexp-f}).\end{proof}

Now we are ready to state and prove the main result of the paper.

\begin{theorem}\label{mainth}
Suppose that the following three conditions are met:
\begin{enumerate}
\item 
$C_0$-semigroup $(e^{tL})_{t\ge 0}$ with generator $(L,D(L))$ in Banach space $\mathcal{F}$ is given, 
such that for some $M_1\geq 1$, $w\geq 0$ and $T>0$ the inequality $\|e^{tL}\| \le M_1e^{wt}$ holds 
for all $t\in [0,T]$. 

\item 
There exists a mapping $S\colon (0,T]\to\mathscr{L}(\mathcal{F})$ 
(i.e. $S(t)\colon \mathcal{F}\to \mathcal{F}$ is a bounded linear operator for each $t\in(0,T]$)
such that for some constant $M_2\geq 1$ the inequality $\|S(t)^k\| \le M_2e^{kwt}$ holds
for all $t\in (0,T]$ and all $k=1,2,3,\dots$.

\item 
Numbers $m\in\{0,1,2,\dots\}$ and $p\in\{1,2,3,\dots\}$ are fixed. There exist a $(e^{tL})_{t\ge 0}$-invariant subspace $\mathcal{D}\subset D(L^{m+p})\subset \mathcal{F}$ (i.e. $e^{tL}(\mathcal{D})\subset \mathcal{D}$ for any $t\geq0$, for example $\mathcal{D}=D(L^{m+p})$ is well suited) and functions $K_j\colon (0,T]\to [0,+\infty)$, $j=0,1,\dots,m+p$ such that we have 
\begin{equation} \label{ocdiffmain}
\bigg\|S(t) f - \sum_{k=0}^{m} \frac{t^kL^k f}{k!}\bigg\| \le t^{m+1}\sum_{j=0}^{m+p}K_j(t)\|L^{j}f\|
\end{equation}
for all $t\in(0,T]$ and all $f\in \mathcal{D}$.
\end{enumerate}

Then the following two statements hold:
\begin{enumerate}
\item 
For all $t>0$, all integer $n\geq t/T$ and all $f\in \mathcal{D}$ the estimate is true:
\begin{equation} \label{ocresdiff1main}
\|S(t/n)^n f - e^{tL}f\|\leq \frac{M_1M_2t^{m+1}e^{wt}}{n^{m}}\sum_{j=0}^{m+p}C_j(t/n)\|L^{j}f\|,
\end{equation}
where $C_{m+1}(t)=K_{m+1}(t)e^{-wt}+M_1/(m+1)!$ and $C_j(t)=K_j(t)e^{-wt}$ 
for all such $j\in\{0,1,\ldots,m+p\}$, that $j\neq m+1$.

\item 
If $\mathcal{D}$ is dense in $\mathcal{F}$ and for all $j=0,1,\dots, m+p$ we have 
$K_j(t)=o(t^{-m})$ as $t\to+0$,
then for all $g\in \mathcal{F}$ and all $\mathcal{T}>0$
the following equality is true:
\begin{equation} \label{eqanChernoff}
\lim_{\mathcal{T}/T\leq n\to\infty}\sup_{t\in(0,\mathcal{T}]}\left\|S(t/n)^ng-e^{tL}g\right\| = 0.
\end{equation}
\end{enumerate}
\end{theorem}

\begin{proof}

1. With the help of estimate (\ref{ocdiffmain}) and lemma~\ref{lemmaTeylocexp} for each $t\in(0,T]$ and each $f\in\mathcal{D}\subset D(L^{m+p})$ we have
\begin{multline*}
\|S(t)f-e^{tL}f\|\leq\bigg\|S(t) f - \sum_{k=0}^{m} \frac{t^kL^k f}{k!}\bigg\|+
\bigg\|\sum_{k=0}^{m} \frac{t^kL^k f}{k!}-e^{tL}f\bigg\|\leq\\
\leq t^{m+1}\sum_{j=0}^{m+p}K_j(t)\|L^{j}f\| + \frac{t^{m+1}}{(m+1)!}\|L^{m+1}f\|\cdot\!\sup_{s\in[0,t]}\left\|e^{sL}\right\|\leq\\
\leq t^{m+1}\left(\sum_{j=0}^{m+p}K_j(t)\|L^{j}f\|+\frac{M_1e^{wt}}{(m+1)!}\|L^{m+1}f\|\right) = 
t^{m+1}\sum_{j=0}^{m+p}e^{wt}C_j(t)\|L^{j}f\|,
\end{multline*}
where $C_{m+1}(t)=K_{m+1}(t)e^{-wt}+M_1/(m+1)!$ and $C_j(t)=K_j(t)e^{-wt}$ for $j\neq m+1$.
Now we see that conditions of lemma~\ref{mainlemmaoc} are satisfied,
so~(\ref{ocresdiff1}) is true with $C_j(t/n)$ replaced by $e^{wt/n}C_j(t/n)$.
Then~(\ref{ocresdiff1main}) follows from~(\ref{ocresdiff1}). Item 1 is thus proved.

2. Suppose that arbitrary $g\in\mathcal{F}$, $\mathcal{T}>0$ and $\varepsilon>0$ are given. 
It is sufficient to find such integer $n_0\geq\mathcal{T}/T$ that 
for all $n>n_0$ and all $t\in(0,\mathcal{T}]$ we have  $\|S(t/n)^ng-e^{tL}g\| < \varepsilon$.

The set $\mathcal{D}$ is dense in $\mathcal{F}$, so for any $\delta>0$ there exists such $f\in\mathcal{D}$ 
that $\|f-g\|<\delta$. Then, using inequality~\eqref{ocresdiff1main} from item 1 proven above we obtain for all $t\in(0,\mathcal{T}]$ and all $n=1,2,3,\ldots$:
\begin{multline*}
\|S(t/n)^ng-e^{tL}g\|\leq\\
\leq\|S(t/n)^ng-S(t/n)^nf\|+\|S(t/n)^nf-e^{tL}f\|+\|e^{tL}f-e^{tL}g\|\leq\\
\leq \|S(t/n)^n\|\cdot\|g-f\|+ \|S(t/n)^nf-e^{tL}f\|+\|e^{tL}\|\cdot \|f-g\|\leq\\
\leq M_2e^{nwt/n}\delta + \frac{M_1M_2t^{m+1}e^{wt}}{n^{m}}\sum_{j=0}^{m+p}C_j(t/n)\|L^{j}f\| + M_1e^{wt}\delta \leq\\
\leq (M_1+M_2)e^{w\mathcal{T}}\delta + M_1M_2\mathcal{T} e^{w\mathcal{T}}\sum_{j=0}^{m+p}(t/n)^mC_j(t/n)\|L^{j}f\|.
\end{multline*}
Let us choose $n_0\geq\mathcal{T}/T$ such that 
$$
M_1M_2\mathcal{T} e^{w\mathcal{T}}\sum_{j=0}^{m+p}(t/n)^mC_j(t/n)\|L^{j}f\|<\varepsilon/2
$$ 
for all $n\geq n_0$ and $t\in(0,\mathcal{T}]$
(such $n_0$ exists due to $\lim_{n\to\infty}C_j(t/n)(t/n)^m=0$ 
thanks to condition $K_j(t)=o(t^{-m})$ when $t\to+0$ for all $j=0,1,\dots, m+p$). Then taking $\delta = \varepsilon e^{-w\mathcal{T}}/(2M_1+2M_2)$
we get:
$\|S(t/n)^ng-e^{tL}g\| < \varepsilon/2+\varepsilon/2 = \varepsilon$
for any $n\geq n_0$ and $t\in(0,\mathcal{T}]$.
The theorem is thus proved.
\end{proof}

\begin{remark}
Condition $\|S(t)^k\| \le M_2e^{kwt}$ may seem difficult to obtain, but if we have the estimate $\|S(t)\| \le e^{wt}$ then $\|S(t)^k\| \le M_2e^{kwt}$ is true for $M_2=1$.
\end{remark}

Let us consider a particular modeling example.
\begin{example} Suppose that $0<\varepsilon<1$ and for all $t\in(0;1]$, all $f\in D(L^3)$ we have $\|e^{tL}\|\leq e^{t}$, $\|S(t)\|\leq e^t$, $\|S(t)f -f-tLf-\frac{1}{2}t^2L^2f\|\leq t^{2+\varepsilon}\|L^3f\|$. Then in the theorem~\ref{mainth} we can take $\mathcal{D}=D(L^3)$, $m=2$, $M_1=M_2=w=1$, $K_0(t)=K_1(t)=K_2(t)=0$,  $K_3(t)=t^{\varepsilon-1}$ for any $t\in(0;1]$. So the estimate (\ref{ocresdiff1main}) of theorem~\ref{mainth} states that for any fixed $t>0$ the following estimate is true for all $f\in D(L^3)$ and all integer $n\geq t$, having the following asymptotic behaviour as $n\to\infty$:
\begin{multline*}
\|S(t/n)^nf-e^{tL}f\|\leq
\frac{t^{3}e^{t}}{n^{2}}\bigg(e^{-t/n}\Big(\frac{t}{n}\Big)^{\varepsilon-1}+\frac{1}{3!}\bigg)\|L^{3}f\| \le\\
\le e^t\bigg(\frac{t^{2+\varepsilon}}{n^{1+\varepsilon}} + \frac{t^{3}}{6n^2}\bigg)\|L^{3}f\| =
\frac{t^{2+\varepsilon}e^{t}}{n^{1+\varepsilon}}\|L^3f\|+O\Big(\frac{1}{n^2}\Big).
\end{multline*}
\end{example}

A more meaningful example of the usage of the theorem~\ref{mainth} can be found in the proof of theorem~\ref{teorApprDU2} in the next section.



\section{Example of application of the main result}

In this section we will show how one can use theorem~\ref{mainth} in practice. In subsection~\ref{UOFEONODVNODOSODO} we consider a second-order parabolic (diffusion type) equation and show that the solution of the Cauchy problem is given by a $C_0$-semigroup.
After that we take one of the known~\cite{RemAMC2018} Chernoff functions for its generator and prove the estimates of speed of convergence of the Chernoff approximations using theorem~\ref{mainth}.

\subsection{First step: estimation of derivatives in terms of a second-order differential operator}
First, we prove theorem~\ref{teorOcVnLk} on estimation of the norms of derivatives of a function via the norms of powers of a second-order differential operator. To do this, we need the following two lemmas.
\begin{lemma} \label{lemOcU1} 
For each twice differentiable function $u\colon{\mathbb R}\to{\mathbb R}$ and any $h>0$, the inequality holds
\begin{equation} \label{eqOcU1}
\sup_{x\in{\mathbb R}}|u'(x)| \le h\cdot\sup_{x\in{\mathbb R}}|u''(x)| + \frac1h\cdot\sup_{x\in{\mathbb R}}|u(x)|.
\end{equation}
\end{lemma}
\begin{proof}
Let us expand the function $u$ using the first-order Taylor formula
at the point $x\in {\mathbb R}$ for the increment $2h$ with remainder in Lagrange form:
$u(x+2h) = u(x)+u'(x)\cdot 2h+u''(\xi)\cdot (2h)^2/2$,
where $\xi\in(x,x+2h)$.
Express the derivative from this formula:
$u'(x) = -u''(\xi)\cdot h + (u(x+2h)-u(x))/(2h)$.
Taking supremums (with respect to $x\in{\mathbb R}$) of the absolute values, we get the estimate~\eqref{eqOcU1}.
\end{proof}

\begin{lemma} \label{lemRazlLq} 
Let $q\in\{1,2,3,\ldots\}$, the functions $a, b,c\colon {\mathbb R}\to{\mathbb R}$ be differentiable $(2q-2)$ times,
the operator $A$ maps every twice differentiable function $u\colon{\mathbb R}\to{\mathbb R}$ to the function $Au = au''+bu'+cu$,
and the function $v\colon{\mathbb R}\to{\mathbb R}$ be differentiable $2q$ times.
Then the following three statements are true:

1) the function $A^q v$ can be written as
\begin{equation} \label{eqRazlLq}
A^q v = a^q v^{(2q)} + \sum_{i=0}^{2q-1} p_i\cdot v^{(i)},
\end{equation}
where functions $p_0,\ldots,p_{2q-1}$ are some homogeneous polynomials of degree $q$ of the functions $a$, $b$, $c$ and their derivatives of order no higher than $2(q-1)$;

2) the following inequality holds:
\begin{equation} \label{eqOcLqVi}
\|A^q v\| \leq \sum_{i=0}^{2q} C_i\cdot\|v^{(i)}\|,
\end{equation}
where $C_i=\|p_i\|$ for $i=0,\ldots,{2q-1}$, and $C_{2q}=\|a\|^q$;

3) in the case $\inf_{x\in{\mathbb R}}|a(x)|>0$ the following estimate is correct:
\begin{equation} \label{eqOcV2qLqVi}
\|v^{(2q)}\| \le \Big\|\frac1{a^q}\Big\|\cdot\|A^q v\| +
\sum_{i=0}^{2q-1} \Big\|\frac{p_i}{a^q}\Big\|\cdot\|v^{(i)}\|.
\end{equation}
\end{lemma}
\begin{proof}
1) The equality~\eqref{eqRazlLq} will be proved by mathematical induction on $q$.

The base case: $q=1$.
In this case, $A^q v = Av = av''+bv'+cv$, so~\eqref{eqRazlLq} is true with $p_0=c$ and $p_1=b$.

Induction step: $q\to q+1$.
Let us assume that the statement 1) of the lemma is
true for the number $q\in\{1,2,3,\ldots\}$ and show that it remains true when replacing $q$ with $q+1$.

Substituting the function $v$ by $Av$ in~\eqref{eqRazlLq}, we get:
\begin{equation*}
A^{q+1}v = A^q(Av) = a^q\cdot(Av)^{(2q)} + \sum_{i=0}^{2q-1} p_i\cdot (Av)^{(i)} =
\end{equation*}
\begin{equation*}
= a^q\cdot\Big((av'')^{(2q)}+(bv')^{(2q)}+(cv)^{(2q)}\Big) + \sum_{i=0}^{2q-1} p_i\cdot\Big((av'')^{(i)}+(bv')^{(i)}+(cv)^{(i)}\Big).
\end{equation*}
Next, using the Leibniz formula
$(uv)^{(i)} = \sum_{j=0}^i C_i^j u^{(i-j)}v^{(j)}$ in each term of the right hand side and selecting separately the first term, we find:
\begin{equation*} 
\begin{split}
A^{q+1}v &= a^{q+1}v^{(2q+2)}
+ \sum_{j=0}^{2q-1} a^q C_{2q}^j a^{(2q-j)}v^{(j+2)} +\\
&+ \sum_{j=0}^{2q} a^q C_{2q}^j\cdot\Big(b^{(2q-j)}v^{(j+1)}
+ c^{(2q-j)}v^{(j)}\Big) +\\
&+ \sum_{i=0}^{2q-1} p_i \sum_{j=0}^i
C_i^j\cdot\Big( a^{(i-j)}v^{(j+2)} +b^{(i-j)}v^{(j+1)} +c^{(i-j)}v^{(j)}\Big).
\end{split}
\end{equation*}
This shows that the function $A^{q+1}v$ can be written in the form similar to~\eqref{eqRazlLq}:
\begin{equation*} 
A^{q+1} v = a^{q+1} v^{(2q+2)} + \sum_{i=0}^{2q+1} r_i\cdot v^{(i)},
\end{equation*}
where functions $r_0,\ldots,r_{2q+1}$ are some homogeneous polynomials of degree $q+1$ of the functions $a$, $b$, $c$ and their derivatives of order no higher than $2q$.
Thus, the induction step is completed and the statement of item 1) of the lemma is proved.

\medskip
2) Inequality~\eqref{eqOcLqVi} immediately follows from the formula~\eqref{eqRazlLq}.

\medskip
3) Expressing the function $v^{(2q)}$ from the equality~\eqref{eqRazlLq} and evaluating its norm, we obtain the required inequality~\eqref{eqOcV2qLqVi}.
\end{proof}

\begin{example}
For $q=2$, the decomposition~\eqref{eqRazlLq} has the following form:
\begin{equation*}
\begin{split}
A^qv &= A^2v = (av''+bv'+cv)''a + (av''+bv'+cv)'b + (av''+bv'+cv)c =\\
&= a^2 v^{IV} + (2aa'+2ab)\cdot v''' + (aa''+a'b+b^2+2ab'+2ac)\cdot v'' +\\
&+ (ab''+bb'+2ac'+2bc)\cdot v' + (ac''+bc'+c^2)\cdot v.
\end{split}
\end{equation*}
\end{example}

The following theorem helps use theorem~\ref{mainth}.
\begin{theorem} \label{teorOcVnLk}
Suppose $n\in\{0,1,2,\ldots\}$, the functions $a,b,c\colon{\mathbb R}\to{\mathbb R}$ are differentiable $2\lfloor(n-1)/2]$ times
and the inequality $\inf_{x\in{\mathbb R}}|a(x)|>0$ holds.
Suppose, in addition, that the operator $A$ maps each twice differentiable function $u\colon{\mathbb R}\to{\mathbb R}$ to the function $Au = au''+bu'+cu$.
Then there exist nonnegative constants $C_0,C_1,\ldots,C_{\lfloor(n+1)/2\rfloor}$, such that for any $2\lfloor(n+1)/2\rfloor$ times differentiable function $v\colon{\mathbb R}\to{\mathbb R}$, the following inequality is true:
\begin{equation} \label{eqOcVnLk}
\|v^{(n)}\| \le \sum_{k=0}^{\lfloor(n+1)/2\rfloor} C_k\|A^k v\|.
\end{equation}
\end{theorem}
\begin{proof}
We apply the induction on parameter $n$.

1) The base case: $n=0$.
In this case~\eqref{eqOcVnLk} has the form $\|v\|\leq C_0\|v\|$,
so we can take $C_0=1$.

2) The induction step. Let the statement of the theorem be proved for all $n\le m-1$. 
We must prove it for $n=m$.

Consider two possible cases: $m$ is odd and $m$ is even.

2.1) Let $m$ be odd. Then putting $u=v^{(m-1)}$ in lemma~\ref{lemOcU1} we have for any $h>0$:
\begin{equation} \label{eqOcVmVm1h}
\|v^{(m)}\| \le h\|v^{(m+1)}\| + \frac{1}{h}\|v^{(m-1)}\|.
\end{equation}
According to item~3) of the lemma~\ref{lemRazlLq} with $q=(m+1)/2$ the following inequality is satisfied for some nonnegative constants $\alpha_0,\ldots, \alpha_m$:
\begin{equation*} 
\|v^{(m+1)}\| \le \Big\|\frac1{a^{(m+1)/2}}\Big\|\cdot\|A^{(m+1)/2}v\| +
\sum_{i=0}^{m} \alpha_i\|v^{(i)}\|.
\end{equation*}
Inserting this inequality into~\eqref{eqOcVmVm1h} we get:
\begin{equation*} 
\begin{split}
\|v^{(m)}\| &\le h\Big\|\frac1{a^{(m+1)/2}}\Big\|\cdot\|A^{(m+1)/2}v\| +
\sum_{i=0}^{m-2} h\alpha_i\|v^{(i)}\| +\\
&+ \Big(h\alpha_{m-1}+\frac1h\Big)\|v^{(m-1)}\| + h\alpha_m\|v^{(m)}\|.
\end{split}
\end{equation*}
From here we have:
\begin{equation} \label{eqOcVmLm12ViPodg}
\begin{split}
(1-h\alpha_m)\|v^{(m)}\| &\le h\Big\|\frac1{a^{(m+1)/2}}\Big\|\cdot\|A^{(m+1)/2}v\| +
\sum_{i=0}^{m-2} h\alpha_i\|v^{(i)}\| +\\
&+ \Big(h\alpha_{m-1}+\frac1h\Big)\|v^{(m-1)}\|.
\end{split}
\end{equation}
Choose $h>0$ so that $1-h\alpha_m>0$
(we can take $h=1$ when $\alpha_m=0$, and $h=1/(2\alpha_m)$ when $\alpha_m>0$).
Then, expressing $\|v^{(m)}\|$ from~\eqref{eqOcVmLm12ViPodg}, we get that for some nonnegative constants
$\beta_0,\ldots,\beta_{m-1}$ the following estimate is correct
\begin{equation} \label{eqOcVmLm12Vi}
\|v^{(m)}\| \le \beta_m\|A^{(m+1)/2}v\| +
\sum_{i=0}^{m-1} \beta_i\|v^{(i)}\|.
\end{equation}
Due to the induction assumption, all values $\|v^{(i)}\|$, $i=0,\ldots,m-1$ that are included into the right-hand side of~\eqref{eqOcVmLm12Vi}, can be evaluated via linear combinations of the values $\|A^k v\|$, $k=0,\ldots, (m-1)/2$.
So, from~\eqref{eqOcVmLm12Vi} it follows that for some nonnegative constants $C_0,\ldots,C_{(m+1)/2}$ the following needed estimate of the type~\eqref{eqOcVnLk} is true:
\begin{equation*} 
\|v^{(m)}\| \le \sum_{k=0}^{(m+1)/2} C_k\|A^k v\|.
\end{equation*}

2.2) Let $m$ be even.
Then, according to item~3) of the lemma~\ref{lemRazlLq} with $q=m/2$ we have for some nonnegative constants $\alpha_0,\ldots,\alpha_{m-1}$:
\begin{equation} \label{eqOcVmLm2Vi}
\|v^{(m)}\| \le \Big\|\frac1{a^{m/2}}\Big\|\cdot\|A^{m/2}v\| +
\sum_{i=0}^{m-1} \alpha_i\|v^{(i)}\|.
\end{equation}
Due to the induction assumption, all values $\|v^{(i)}\|$, $i=0,\ldots,m-1$ that are included into the right-hand side of~\eqref{eqOcVmLm2Vi}, can be evaluated via linear combinations of the values $\|A^k v\|$, $k=0,\ldots,m/2$. So, from~\eqref{eqOcVmLm2Vi} it follows that for some nonnegative constants $C_0,\ldots,C_{m/2}$ the following needed estimate of the type~\eqref{eqOcVnLk} is true:
\begin{equation*} 
\|v^{(m)}\| \le \sum_{k=0}^{m/2} C_k\|A^k v\|.
\end{equation*}

So the induction step is done for both cases ($m$ is odd and $m$ is even).
Then the inequality~\eqref{eqOcVnLk}, and with it the whole theorem~\ref{teorOcVnLk} are proved.
\end{proof}


\subsection{Second step: estimation of the rate of convergence of Chernoff approximations to solution of second-order parabolic PDEs}
\label{UOFEONODVNODOSODO}

Now, using theorems \ref{mainth} and~\ref{teorOcVnLk}, as well as the results of the book~\cite{Krylov}, 
we prove a theorem on the approximation of solutions to the Cauchy problem for second-order parabolic partial differential equations (PDEs) via the Chernoff function.

Recall that we use notation $UC_b(\mathbb{R})$ for the Banach space of all real-valued bounded uniformly continuous functions on $\mathbb{R}$.
Similarly, $UC_b^n(\mathbb{R})$ denotes the space of all such functions $u\in UC_b(\mathbb{R})$, that $u',\ldots,u^{(n)}\in UC_b(\mathbb{R})$.
Let us denote by symbol $HC_b(\mathbb{R})$ the space of all H\"older continuous functions $u\colon\mathbb{R}\to\mathbb{R}$.
Next, for each $n\in\{1,2,3,\ldots\}$ let us denote by symbol $HC_b^n(\mathbb{R})$ the space of all such functions $u\in HC_b(\mathbb{R})$, that $u',\ldots,u^{(n)}\in HC_b(\mathbb{R})$.
We denote by symbol $C_b^\infty(\mathbb{R})$ the space of all real-valued functions on $\mathbb{R}$ bounded with all derivatives.

\begin{remark} \label{remUCnDenseUC}
It is clear that $C_b^\infty(\mathbb{R})\subset HC_b^n(\mathbb{R})\subset UC_b^n(\mathbb{R})\subset UC_b(\mathbb{R})$. So the spaces $HC_b^n(\mathbb{R})$ and $UC_b^n(\mathbb{R})$ are dense in $UC_b(\mathbb{R})$, because $C_b^\infty(\mathbb{R})$ is dense in $UC_b(\mathbb{R})$ (this is proved e.g. in~\cite[lemma 1]{RemJMP}).
\end{remark}

\begin{theorem} \label{teorApprDU2}
Suppose that the following three conditions are met:
\begin{enumerate}
\item 
Numbers $m,q\in\{1,2,3,\dots\}$ are  fixed, and $\hat{q}=2\lfloor(q+1)/2\rfloor$.
Functions $a,b,c$ from the class $HC_b^{2m+\hat{q}-2}(\mathbb{R})$ are given, 
such that $\inf_{x\in{\mathbb R}}a(x)>0$. 
Operator $A$ on $UC_b(\mathbb{R})$ with domain $D(A)=HC_b^2(\mathbb{R})$ is defined by the formula
$Au = au''+bu'+cu$.

\item 
Numbers $T>0$, $M\geq 1$ and $\sigma\geq0$ are given. For any $t\in(0,T]$ bounded linear operator $S(t)$ on $UC_b(\mathbb{R})$ is defined
such that $\|S(t)^k\| \le Me^{k\sigma t}$ for any $k=1,2,3,\dots$.

\item 
There exist constant $\alpha\le1$ and nonnegative constants $K_0,K_1,\dots,K_{2m+q}$ such that for all $t\in(0,T]$ and all $f\in UC_b^{2m+q}(\mathbb{R})$ we have 
\begin{equation} \label{eqApprDU2_1}
\bigg\|S(t)f - \sum_{k=0}^{m}\frac{t^kA^kf}{k!}\bigg\| \le t^{m+\alpha}\sum_{i=0}^{2m+q}B_i\|f^{(i)}\|.
\end{equation}
\end{enumerate}

Then the following three statements hold:
\begin{enumerate}
\item 
The closure $\overline{A}$ of operator $A$ in Banach space $UC_b(\mathbb{R})$ is a generator of $C_0$-semigroup  $(e^{t\overline{A}})_{t\ge 0}$ in $UC_b(\mathbb{R})$, and the condition $\|e^{t\overline{A}}\|\le e^{\gamma t}$ for all $t\geq0$ is satisfied, where $\gamma=\sup_{x\in\mathbb{R}}c(x)$.

\item 
Let $w=max(\sigma,\gamma,0)$. Then nonnegative constants $C_0,C_1,\ldots,C_{2m+\hat{q}}$ exists (which are independent of $t$, $T$ and $n$) such that for all $t>0$, all  integer $n\geq n_{\alpha,t}$ (where $n_{\alpha,t}=t/T$ if $\alpha=1$ and $n_{\alpha,t}=max(t/T,t)$ if $\alpha<1$) and all $f\in UC_b^{2m+\hat{q}}(\mathbb{R})$ we have
\begin{equation} \label{eqApprDU2_2}
\big\|S(t/n)^n f - e^{t\overline{A}}f\big\|\leq \frac{Mt^{m+\alpha}e^{wt}}{n^{m-1+\alpha}}\sum_{i=0}^{2m+\hat{q}}C_i\|f^{(i)}\|.
\end{equation}

\item 
If $\alpha>1-m$ then for all $\mathcal{T}>0$ and all $g\in UC_b(\mathbb{R})$ the following equality is true:
\begin{equation} \label{eqApprDU2_3}
\lim_{\mathcal{T}/T\leq n\to\infty}\sup_{t\in(0,\mathcal{T}]}\big\|S(t/n)^ng-e^{t\overline{A}}g\big\| = 0.
\end{equation}
\end{enumerate}
\end{theorem}
\begin{proof}
1). The proof of the first statement is divided into two parts.

1.1) First, assume that $c(x)\leq0$ for all $x\in\mathbb{R}$.
Then because of~\cite[theorem~8.2.1 on p.~111 and corollary~8.3.1 on p.~114]{Krylov} for any function $f\in HC_b(\mathbb{R})$ the equation $u'_t(t,\cdot)=Au(t,\cdot)$, $t>0$ has a
unique solution $u(t,\cdot)=u_f (t,\cdot)\in HC_b^2 (\mathbb{R})$, $t>0$,
such that $\lim_{t\to+0}\|u_f(t,\cdot)-f\| = 0$. 
Moreover, for all $t>0$ the inequality $\|u_f(t,\cdot)\|\leq\|f\|$ is satisfied.
Based on the above, for any function $f\in HC_b(\mathbb{R})$ let $Q(t)f=u_f(t,\cdot)$ if $t>0$ and $Q(t)f=f$ if $t=0$.
Thus, the following relations will be fulfilled
\begin{equation}\label{eqDifTtfIsLTtf} 
(Q(t)f)'_t = A(Q(t)f) \quad\text{for all }t>0\text{ and all }f\in HC_b(\mathbb{R}),
\end{equation}
\begin{equation}\label{eqHC2bInvDifTt} 
Q(t)f\in HC^2_b(\mathbb{R}) \quad\text{for all }t>0\text{ and all }f\in HC_b(\mathbb{R}),
\end{equation}
\begin{equation}\label{eqLim0TtfIsf} 
\lim_{t\to+0} Q(t)f = f \quad\text{for all }f\in HC_b(\mathbb{R}).
\end{equation}
Since the solution $u_f$ is unique, then for all $t,s\geq0$, and all $f\in HC_b(\mathbb{R})$ the semigroup property $Q(t+s)f=Q(t)Q(s)f$ holds.
So, $(Q(t))_{t\geq0}$ is a $C_0 $-semigroup on the space $HC_b (\mathbb{R})$, 
with estimate $\|Q(t)\|\leq1$ for all $t\geq0$.

Due to $HC_b(\mathbb{R})$ being dense in $UC_b(\mathbb{R})$ (see remark~\ref{remUCnDenseUC}), the operators $Q(t)$ for any $t\geq0$ can be continued by continuity over the whole space $UC_b (\mathbb{R})$, preserving the norm.
So we get that $(Q(t))_{t\geq0}$ is a contraction $C_0$-semigroup on the space $UC_b (\mathbb{R})$.

Let's show that the generator $L$ of the semigroup $(Q(t))_{t\geq0}$ coincides with the closure $\overline{A}$ of the operator $A$.
To do this, first recall (see definition~1.2 or~\cite[lemma~1.1]{EN}), that
$D(L) = \{\varphi\in UC_b(\mathbb{R}) \,\big|\, \lim_{s\to +0}(Q(s)\varphi-\varphi)/s \text{ exists}\}$, 
and $L\varphi = \lim_{s\to +0}(Q(s)\varphi-\varphi)/s$ for any $\varphi\in D(L)$.
Let $f\in HC_b(\mathbb{R})$ and $t>0$. Then by virtue of semigroup property and the equality~\eqref{eqDifTtfIsLTtf} we have:
$$
\lim_{s\to +0} \frac{Q(s)Q(t)f-Q(t)f}{s} = \lim_{s\to +0} \frac{Q(t+s)f-Q(t)f}{s} =
$$
$$
= (Q(t)f)'_t \stackrel{by\;(\ref{eqDifTtfIsLTtf})}{=} A(Q(t)f).
$$
From this it follows that 
\begin{equation}\label{eqTtfInDA} 
Q(t)f\in D(L),\ L(Q(t)f)=A(Q(t)f) \text{ for all }t>0,f\in HC_b(\mathbb{R}).
\end{equation}

Now let us assume that $f\in HC^2_b(\mathbb{R}) \subset HC_b(\mathbb{R})$. Thanks to~\cite[remark~8.3.2 on p.~114]{Krylov} we have:
$A(Q(t)f) = [Q(t)](Af)$ for all $t>0$.
Then from this and from the formulas \eqref{eqTtfInDA}, \eqref{eqLim0TtfIsf} the continued equality follows:
$$
\lim_{n\to\infty} L(Q(1/n)f) \stackrel{by\;(\ref{eqTtfInDA})}{=}
\lim_{n\to\infty} A(Q(1/n)f) = \lim_{n\to\infty} [Q(1/n)](Af) \stackrel{by\;(\ref{eqLim0TtfIsf})}{=} Af.
$$
So for each $f\in HC^2_b(\mathbb{R})$ the following two relations are correct:
$$
\lim_{n\to\infty} L(Q(1/n)f) = Af\quad \text{and}\quad\lim_{n\to\infty} Q(1/n)f = f.
$$
Since the generator $L$ is closed~\cite[theorem~1.4]{EN}, it follows that $f\in D(L)$ and $Lf=Af$.
Thus, the restriction of the operator $L$ to the subspace $HC^2_b(\mathbb{R})$ matches the operator $A$, i.e. $L|_{HC^2_b(\mathbb{R})} = A$.

Subspace $HC^2_b(\mathbb{R})$ is invariant under the semigroup $(Q(t))_{t\geq0}$ (by virtue of~\eqref{eqHC2bInvDifTt}), 
and is dense in $UC_b(\mathbb{R})$. Therefore according to~\cite[prop.~1.7 of ch.~2]{EN}, 
subspace $HC^2_b(\mathbb{R})$ is the core of the generator $L$. According to the definition of the core (see definition~\ref{defCore} or~\cite[def.~1.6 of ch.~2]{EN}) this means that $HC^2_b(\mathbb{R})$ is dense in $D(L)$ for the graph norm $\|x\|_L=\|x\|+\|Lx\|$.
From this and from the equality $L|_{HC^2_b(\mathbb{R})} = A$ it follows that $L=\overline{A}$.

1.2) Let us proceed to the general case, where the function $c(x)$ can have its sign changed.
By virtue of the equality $\gamma=\sup_{x\in\mathbb{R}}c(x)$ we have $c(x)-\gamma\leq0$ for all $x\in\mathbb{R}$.
Using the results of item~1.1) of the proof for linear operator $(A-\gamma)u = au''+bu'+(c-\gamma)u$, 
we get that the closure $\overline{A-\gamma} = \overline{A}-\gamma$ of operator $A-\gamma$ 
is the generator of $C_0$-semigroup  $(e^{t(\overline{A}-\gamma)})_{t\ge 0}$ in Banach space $UC_b(\mathbb{R})$, 
and the condition $\|e^{t(\overline{A}-\gamma)}\|\le 1$ holds for all $t\geq0$.
Hence the operator $\overline{A}$ is the generator of $C_0$-semigroup  
$(e^{t\overline{A}})_{t\ge 0} = (e^{\gamma t}\cdot e^{t(\overline{A}-\gamma)})_{t\ge 0}$, 
and the condition $\|e^{t\overline{A}}\|\le e^{\gamma t}$ holds for all $t\ge 0$.
So the first statement of the theorem is proved.

2). 
It follows from theorem~\ref{teorOcVnLk} that for any $i=0,\ldots,2m+q$ there exist nonnegative constants $C_{i,0},C_{i,1},\ldots,C_{i,\lfloor(i+1)/2\rfloor}$, such that for any $f\in UC_b^{2m+\hat{q}}(\mathbb{R})$ we have
\begin{equation*} 
\|f^{(i)}\| \le \sum_{j=0}^{\lfloor(i+1)/2\rfloor} C_{i,j}\|A^j f\|.
\end{equation*}
From this and from the inequality~\eqref{eqApprDU2_1}, the relations follow:
\begin{equation*}
\bigg\|S(t)f - \sum_{k=0}^{m} \frac{t^kA^k f}{k!}\bigg\| \le t^{m+\alpha}\sum_{i=0}^{2m+q}\sum_{j=0}^{\lfloor(i+1)/2\rfloor}B_iC_{i,j}\|A^jf\|=
\end{equation*}
\begin{equation}\label{eqTeilAlfa}
= t^{m+\alpha}\sum_{j=0}^{m+\lfloor(q+1)/2\rfloor}\alpha_j\|A^jf\| = t^{m+1}\sum_{j=0}^{m+\lfloor(q+1)/2\rfloor}K_j(t)\|A^jf\|,
\end{equation}
where $\alpha_j$ is some nonnegative constant and $K_j(t)=\alpha_jt^{\alpha-1}$ for any $j=0,1,\ldots,m+\lfloor(q+1)/2\rfloor$.

Subspace $UC_b^{2m+\hat{q}}(\mathbb{R})$ is $(Q(t))_{t\geq0}$-invariant due to~\eqref{eqHC2bInvDifTt},
and is dense in $UC_b(\mathbb{R})$ due to remark~\ref{remUCnDenseUC}. 
So, taking into account item 2 of condition and proved statement~1 of this theorem, we see that all the conditions of theorem~\ref{mainth} are met with $\mathscr{D}=UC_b^{2m+\hat{q}}(\mathbb{R})$.
Then, it follows from the inequality~\eqref{ocresdiff1main} of theorem~\ref{mainth} that for any $t>0$, any $n\geq t/T$ and any $f\in UC_b^{2m+\hat{q}}$ we have
\begin{equation*} 
\|S(t/n)^n f - e^{t\overline{A}}f\|\leq \frac{Mt^{m+1}e^{wt}}{n^{m}}\sum_{j=0}^{m+\lfloor(q+1)/2\rfloor}\beta_j(t/n)\|A^{j}f\|,
\end{equation*}
where $\beta_j(t)=K_j(t)e^{-wt}\leq \alpha_jt^{\alpha-1}$ for $j\neq m+1$ and $\beta_{m+1}(t)=K_{m+1}(t)e^{-wt}+1/(m+1)! \leq \alpha_{m+1}t^{\alpha-1}+1/(m+1)!$. It is clear that if $\alpha=1$ or $t\le1$ then $\beta_{m+1}(t)\leq (\alpha_{m+1}+1/(m+1)!)t^{\alpha-1}$.
Consequently, for any $t>0$ and any integer $n\ge n_{\alpha,t}$ the following inequality is true:
\begin{equation*} 
\|S(t/n)^n f - e^{t\overline{A}}f\|\leq \frac{Mt^{m+\alpha}e^{wt}}{n^{m-1+\alpha}}\sum_{j=0}^{m+\lfloor(q+1)/2\rfloor}\gamma_j\|A^{j}f\|,
\end{equation*}
where $\gamma_j=\alpha_j$ for $j\neq m+1$ and $\gamma_{m+1}(t)= \alpha_{m+1}+1/(m+1)!$.

From this and from item 2) of lemma~\ref{lemRazlLq}, 
it follows that for some nonnegative constants $C_0,C_1,\ldots,C_{2m+\hat{q}}$ which are independent of $t$, $T$ and $n$, 
the inequality~\eqref{eqApprDU2_2} that we are proving holds:
\begin{equation*} 
\|S(t/n)^n f - e^{t\overline{A}}f\|\leq \frac{Mt^{m+\alpha}e^{wt}}{n^{m-1+\alpha}}\sum_{i=0}^{2m+\hat{q}}C_i\|f^{(i)}\|.
\end{equation*}

3).
Equality~\eqref{eqApprDU2_3} follows from the estimate~\eqref{eqTeilAlfa},
from the relations $K_j(t)=\alpha_jt^{\alpha-1}=o(t^{-m})$ as $t\to+0$ for all $j=0,1,\ldots,m+\lfloor(q+1)/2\rfloor$, and from the equality~\eqref{eqanChernoff} in the statement~2 of the theorem~\ref{mainth}.
\end{proof}

Here is an example of using theorem~\ref{teorApprDU2} for one concrete Chernoff function, which was presented in~\cite{RemAMC2018}.

\begin{example}\label{exremizf}
Suppose that functions $a,b,c\in HC_b^2(\mathbb{R})$ are given such that $\inf_{x\in{\mathbb R}}a(x)>0$.
For each $u\in UC_b^2(\mathbb{R})$ set
\begin{equation} \label{eqExamp3} 
Au = au'' + bu' + cu
\end{equation}
and for each $t\geq0$, each $f\in UC_b(\mathbb{R})$ and each $x\in\mathbb{R}$ set
\begin{equation} \label{eqExamp2} 
\begin{split}
(S(t)f)(x) = \frac14 f\Big(x+2\sqrt{a(x)t}\Big) + \frac14 f\Big(x-2\sqrt{a(x)t}\Big) +\\
+ \frac12 f\big(x+2b(x)t\big) + tc(x)f(x).
\end{split}
\end{equation}
Then there exist nonnegative constants $C_0,C_1,\ldots,C_4$ such that 
for all $t>0$, all $n\in\{1,2,3,\ldots\}$ and all $f\in UC_b^4(\mathbb{R})$ the following inequality holds:
\begin{equation*} 
\begin{split}
&\| S(t/n)^n f - e^{t\overline{A}}f \| \leq\\
&\leq \frac{t^2e^{\|c\|t}}{n}\big(C_0\|f\|+C_1\|f'\|+C_2\|f''\|+C_3\|f'''\|+C_4\|f^{(IV)}\|\big).
\end{split}
\end{equation*}
\end{example}
\begin{proof}
1) Set $m=1$, $q=2$. Then $\hat{q}=2$ and item 1 of the condition of the theorem~\ref{teorApprDU2} is met.

2) Let us estimate the norm $\|S(t)f\|$ for any $t>0$ and any $f\in UC_b(\mathbb{R})$ using the formula~\eqref{eqExamp2}:
$$
\|S(t)f\| \leq \frac14\sup_{x\in\mathbb{R}}\Big|f\Big(x+2\sqrt{a(x)t}\Big)\Big| + 
\frac14 \sup_{x\in\mathbb{R}}\Big|f\Big(x-2\sqrt{a(x)t}\Big)\Big| +
$$
$$
+ \frac12\sup_{x\in\mathbb{R}}\big|f\big(x+2b(x)t\big)\big| + t\sup_{x\in\mathbb{R}}|c(x)|\cdot\sup_{x\in\mathbb{R}}|f(x)| \le
$$
$$
\le \frac14\|f\| + \frac14\|f\| + \frac12\|f\| + t\|c\|\cdot\|f\| = (1+t\|c\|)\cdot\|f\| \leq e^{\|c\|t}\|f\|.
$$
So $\|S(t)\|\leq e^{\|c\|t}$ and $\|S(t)^k\|\leq e^{k\|c\|t}$ for any $t>0$ and any $k\in\{1,2,3,\ldots\}$. 
Then item 2 of the condition of the theorem~\ref{teorApprDU2} is met with $M=1$, $\sigma=\|c\|$ and any $T>0$.

3) Let's take any function $f\in UC_b^{4}(\mathbb{R})$ and expand $[S(t)f](x)$ in powers of $t>0$, 
using Taylor’s formula with remainders in Lagrange’s form.
Then we have for some real $\xi_1=\xi_1(t,x)$, $\xi_2=\xi_2(t,x)$ and $\xi_3=\xi_3(t,x)$:
$$
f\Big(x+2\sqrt{a(x)t}\Big) = f(x) + f'(x)\cdot 2\sqrt{a(x)t} + 
\frac12 f''(x)\cdot\Big(2\sqrt{a(x)t}\Big)^2 + 
$$
$$
+ \frac16 f'''(x)\cdot\Big(2\sqrt{a(x)t}\Big)^3 +
\frac{1}{24}f^{IV}(\xi_1)\cdot\Big(2\sqrt{a(x)t}\Big)^4;
$$ 
$$
f\Big(x-2\sqrt{a(x)t}\Big) = f(x) - f'(x)\cdot 2\sqrt{a(x)t} + 
\frac12 f''(x)\cdot\Big(2\sqrt{a(x)t}\Big)^2 -
$$
$$ 
- \frac16 f'''(x)\cdot\Big(2\sqrt{a(x)t}\Big)^3 +
\frac{1}{24}f^{IV}(\xi_2)\cdot\Big(2\sqrt{a(x)t}\Big)^4;
$$ 
$$
f(x+2b(x)t) = f(x) + f'(x)\cdot 2b(x)t + 
\frac12 f''(\xi_3)\cdot(2b(x)t)^2.
$$ 
Therefore, using these three equalities together with~\eqref{eqExamp2}
we get the following expression for $[S(t)f](x)$:
\begin{equation} 
\begin{split}
[S(t)f](x) = 
f(x) + t[a(x)f''(x) + b(x)f'(x) +c(x)f(x)] +\\
+ t^2\Big[\frac{(a(x))^2}{6}\big(f^{IV}(\xi_1) + f^{IV}(\xi_2)\big) + (b(x))^2f''(\xi_3)\Big].
\end{split}
\end{equation}
So, taking into account the formula~\eqref{eqExamp3}, we come to the inequality
$$
\|S(t)f - (f + tAf)\| \leq t^2\Big(\frac{\|a\|^2}{3}\|f^{IV}\| + \|b\|^2\|f''\|\Big).
$$
Then last item (item 3) of the condition of the theorem~\ref{teorApprDU2} is met with $\alpha=1$.

\medskip
4) Further, using item~2 of the asserting part of the theorem~\ref{teorApprDU2}, we get that for all $t>0$,  all $n=1,2,3,\ldots$ and all $f\in UC_b^{4}(\mathbb{R})$ the estimate
\begin{equation*} 
\begin{split}
&\| S(t/n)^n f - e^{t\overline{A}}f \| \leq\\
&\leq \frac{t^2e^{\|c\|t}}{n}\big(C_0\|f\|+C_1\|f'\|+C_2\|f''\|+C_3\|f'''\|+C_4\|f^{(IV)}\|\big)
\end{split}
\end{equation*}
is true for some nonnegative constants $C_0$, $C_1$, $C_2$, $C_3$, $C_4$.
\end{proof}

\textbf{Acknowledgements.} Authors are partially supported by the Laboratory of Dynamical Systems and Applications NRU HSE, and by the Ministry of Science and Higher Education of the RF grant ag.~No~075-15-2019-1931. Authors are thankful to the members of research group ``Evolution semigroups and applications'' and Professor Dmitry Turaev for the discussion of the research presented.



\end{document}